\documentclass[12pt]{amsart}

\usepackage{amssymb, enumerate, hyperref}

\makeatletter
\@namedef{subjclassname@2020}{%
  \textup{2020} Mathematics Subject Classification}
\makeatother

\usepackage[T1]{fontenc}

\newtheorem{theorem}{Theorem}[section]
\newtheorem{corollary}[theorem]{Corollary}
\newtheorem{lemma}[theorem]{Lemma}

\theoremstyle{definition}

\newtheorem{remark}[theorem]{Remark}
\newtheorem{rems}[theorem]{Remarks}

\numberwithin{equation}{section}

\def\R{\mathbb{R}}
\def\T{\mathbb{T}}
\def\D{\mathbb{D}}
\def\C{\mathbb{C}}
\def\N{\mathbb{N}}
\def\A{\mathcal{A}}
\def\B{\mathcal{B}}
\def\a{\alpha}
\def\b{\beta}
\def\l{\lambda}
\def\LL{\mathcal{L}L^1}
\def\Re#1{\operatorname{Re}#1}

\begin{document}


\baselineskip=17pt


\title[A continuous-parameter Katznelson--Tzafriri theorem]{A continuous-parameter Katznelson--Tzafriri theorem for algebras of analytic functions}

\author{Charles Batty}
\address{St. John's College\\
University of Oxford\\
Oxford OX1 3JP, UK}
\email{charles.batty@sjc.ox.ac.uk}

\author{David Seifert}
\address{School of Mathematics, Statistics and Physics\\
 Newcastle University\\
 Newcastle upon Tyne, NE1 7RU, UK}
\email{david.seifert@ncl.ac.uk}

\date{}

\begin{abstract}
We prove a continuous-parameter version of the recent theorem of Katznelson--Tzafiri type for power-bounded operators which have a bounded calculus for analytic Besov functions.    We also show that the result can be extended to some operators which have functional calculi with respect to some larger algebras.
\end{abstract}



\maketitle

\section{Introduction}

In 1986, Katznelson and Tzafriri \cite{KT86} proved a theorem concerning asymptotics of the discrete semigroup $(T^n)_{n\ge0}$ for a power-bounded operator $T$ on a complex Banach space $X$.    They showed that $\lim_{n\to\infty} \|T^n(I-T)\| = 0$ if $\sigma(T) \cap \T \subseteq \{1\}$.    More generally, they considered functions in the Wiener algebra $W^+(\D)$ of the form $f(z) = \sum_{k=0}^\infty a_kz^k$, where $\sum_{k=0}^\infty |a_k| < \infty$ and $|z|\le1$.   Let $f(T) = \sum_{k=0}^\infty a_kT^k$.    Let $W(\T)$ be the space of all functions on $\T$ of the form $g(z) = \sum_{k=-\infty}^\infty b_k z^k$, where $\|g\|_{W(\T)}:=\sum_{k=-\infty}^\infty |b_k|<\infty$.   It was shown in \cite{KT86} that $\lim_{n\to\infty} \|T^nf(T)\|=0$ if $f \in W^+(\D)$ and $f$ is of spectral synthesis in $W(\T)$ with respect to $\sigma(T) \cap \T$.   This assumption means that there exist functions $(g_k)_{k\ge1}$ in $W(\T)$ such that each $g_k$ vanishes on a neighbourhood $U_k$ of $\sigma(T) \cap \T$ in $\T$ and $\lim_{k\to\infty} \|g_k-f\|_{W(\T)}=0$.   

These theorems have had a variety of applications.  For a selection of them, see Section 4 of the recent survey article \cite{BS21}.   One drawback of the more general theorem is the assumption of spectral synthesis, which is stronger than simply assuming that $f$ vanishes on $\sigma(T) \cap \T$.  The theorem is not true in general if the weaker assumption is used, but the two assumptions are equivalent if $\sigma(T) \cap \T$ is countable, for example.   If $X$ is a Hilbert space, then the weaker assumption is sufficient \cite{Lek09}.    The survey article \cite{BS21} covers these and other variants of the theorem, mainly in the discrete case. 

The following theorem is an analogue of the original theorem for bounded $C_0$-semigroups, and it was proved in \cite{ESZ2} and \cite{Vu92}.  Their proofs were quite different from the proofs in \cite{KT86} and from each other.    For a discussion of other proofs, see \cite[Section 3.1]{BS21}.  In this paper we shall closely follow the approach used in \cite{Vu92}.

\begin{theorem}  \label{KT}
Let $-A$ be the generator of a bounded $C_0$-semigroup $(T(t))_{t\ge0}$ on a Banach space $X$.   Let $f$ be the Laplace transform of a function in $L^1(\R_+)$, and assume that $f$ is of spectral synthesis in $L^1(\R)$ with respect to $\sigma(A) \cap i\R$.   Then $\lim_{t\to\infty} \|T(t)f(A)\| = 0$.
\end{theorem}

Some variants of this continuous-parameter version of the Katznelson--Tzafriri theorem have been obtained.   If $X$ is a Hilbert space, the  result is true with the weaker assumption that $f$ vanishes on $\sigma(A) \cap i\R$ \cite{Sei15}.   Several extensions to more general semigroups of operators have appeared in \cite{Sei15} and \cite{Zar13}.   Some papers have considered Banach algebras other than the Laplace transforms of integrable functions.   An example in the discrete case is a Banach algebra $\B(\D)$ of analytic Besov functions on the unit disc $\D$, and Peller showed in \cite{Pel82} that any power-bounded operator on a Hilbert space has a bounded $\B(\D)$-calculus.   In the continuous-parameter case, the theory of functional calculus for the corresponding algebra $\B$ of analytic Besov functions on the right half-plane $\C_+$ and bounded $C_0$-semigroups has recently been developed in \cite{BGT} and \cite{BGT2}.   

In \cite{BS22}, we proved a version of the Katznelson--Tzafriri theorem in the discrete case where $T$ is assumed to have a bounded functional calculus with respect to the Banach algebra $\B(\D)$.  It applies to functions $f \in \B(\D)$ which vanish on $\sigma(T) \cap \T$, so it includes not only some functions outside $W^+(\D)$, but also functions in $W^+(\D)$ which are not of spectral synthesis with respect to $\sigma(T) \cap \T$.  When we wrote \cite{BS22}, we did not see how to obtain a corresponding result in the continuous-parameter case.   Subsequently we have been able to prove the result stated in Theorem~\ref{KTB} below, by using an  approximation argument  based on Arveson's  theory of spectral subspaces for $C_0$-groups of isometries; see  \cite{Sei15} for a related argument. We also rely crucially on results by Cojuhari and Gomilko \cite{CG08} concerning a certain integral resolvent condition which characterises those operators that admit a bounded $\B$-calculus.

We now state the main result of our paper for operators which have a bounded $\B$-calculus.    The assumption that $f$ vanishes at infinity is natural in the result because infinity can be thought of as being an invisible part of $\sigma(A)$ when $A$ is unbounded;  indeed without this assumption the theorem would be false if  $\sigma(A) \cap i\R$ is empty, $(T(t))_{t\ge0}$ is not exponentially stable, and $f$ is a constant function.    In the statement, the operator $f(A)$ is defined by the $\B$-calculus for $A$.  

\begin{theorem} \label{KTB}
Let $-A$ be the generator of a bounded $C_0$-semigroup $(T(t))_{t\ge0}$ and assume that $A$ has a bounded $\B$-calculus.   Let $f \in \B$, and assume that $f$ vanishes on $\sigma(A) \cap i\R$ and $ \lim_{z\in\C_+, \,|z|\to\infty} f(z) = 0$.  Then
\[
\lim_{t\to\infty} \|T(t)f(A)\| = 0.
\]
\end{theorem}

Since Laplace transforms of functions in $L^1(\R_+)$ lie in $\B$ and vanish at infinity, this result extends Theorem~\ref{KT} in the case when $A$ admits a bounded $\B$-calculus, both by enlarging the class of functions to which it applies and by weakening the spectral synthesis condition to the condition that $f$ vanishes on $\sigma(A) \cap i\R$ (which is necessary for the conclusion of the theorem). Furthermore,  as the negative generator of every bounded $C_0$-semigroup on a Hilbert space admits a bounded $\B$-calculus (see \cite[Section~4]{BGT}), we obtain the following result as an immediate consequence of Theorem~\ref{KTB}.

\begin{corollary}  \label{KTH}
Let $-A$ be the generator of a bounded $C_0$-semigroup $(T(t))_{t\ge0}$ on a Hilbert space. Let $f \in \B$, and assume that $f$ vanishes on $\sigma(A) \cap i\R$ and $ \lim_{z\in\C_+,\,|z|\to\infty} f(z) = 0$.  Then $\lim_{t\to\infty} \|T(t)f(A)\| = 0.$
\end{corollary}

This extends the implication (i)$\implies$(iii) of \cite[Theorem~3.1]{Sei15} in the case when $S=\R_+$ to a larger class of functions $f$; see also \cite{Lek09}. 

In the next section, we recall relevant facts for understanding the statement and proof of Theorem \ref{KTB}.   We originally considered only the algebra $\B$, but an anonymous referee suggested that we might be able to prove Theorem \ref{KTB} for algebras other than $\B$ by similar methods.   After some time, we proved a version of the result for some algebras larger than $\B$, and that version is set out in Theorem \ref{KTA}.    The proof is given in Section \ref{proof}.  

The papers \cite{ALM} and \cite{ALMZ} have identified two function algebras which are larger than $\B$, and to which Theorem  \ref{KTA} can be extended in a slightly modified form.   These algebras are briefly described in Section \ref{exs}.

\section{The setting and the main result}\label{sec2}

Let $\B$ be the space of all holomorphic functions $f$ on $\C_+$ such that
\[
\|f\|_{\B_0} := \int_0^\infty \sup_{\b\in\R} |f'(\a+i\b)| \, d\a < \infty.
\]
These functions are bounded and uniformly continuous on $\C_+$, and $\B$ is a Banach algebra in the norm 
\[
\|f\|_\B := \|f\|_{\B_0} + \|f\|_\infty,
\]
where $\|\cdot\|_\infty$ is the supremum norm on $H^\infty(\C_+)$.   We will consider any function $f \in \B$ to be defined and uniformly continuous on $\overline{\C_+}$.    For details about $\B$, see \cite{BGT}.  

We will also consider the following subalgebras of $\B$, adopting the notation used in \cite{BGT}, \cite{BGT2} and \cite{BGTs}:
\begin{align*}
\LL &:= \left\{ \L g: g \in L^1(\R_+) \right\}, \quad \text{where $\L$ is the Laplace transform}, \notag \\
\B_{00} &:=\Big \{f \in \B:  \lim_{|\b|\to\infty} f(i\b) = 0 \Big\} = \Big\{f \in \B:  \lim_{{z\in\C_+},\,{|z|\to\infty}} f(z) = 0 \Big\}.  \label{B00}
\end{align*}
 The algebra $\LL$ is dense in $\B_{00}$ in the $\B$-norm \cite[Theorem 4.4]{BGT2}, and $\B_{00}$ is a closed subalgebra of $\B$, so $\B_{00}$ is the closure of $\LL$ in $\B$.    The norms $\|\cdot\|_{\B_0}$ and $\|\cdot\|_\B$ are equivalent on $\B_{00}$, but $\|\cdot\|_{\B_0}$ is not submultiplicative.   Note that the space $\B_{00}$ is denoted by $\B \cap C_0(\overline{\C_+})$ in \cite{BGT2}.

Let $-A$ be the generator of a bounded $C_0$-semigroup on a Banach space $X$, so that $\sigma(A) \subseteq \overline{\C_+}$.   We say that $A$ {\it satisfies the {\rm (GSF)} condition} if, for all $x \in X$ and $x^* \in X^*$,
\begin{equation} \label{GSF}
\sup_{\a>0} \,\a \int_\R \left| \langle (\a+i\b+A)^{-2}x,x^* \rangle \right| \, d\b < \infty.
\end{equation}
The Closed Graph Theorem then implies that there exists a finite constant $\gamma_A$ such that the supremum above is bounded by $\gamma_A \,\|x\|\,\|x^*\|$ for all $x\in X$ and $x^* \in X^*$.

Let $L(X)$ be the Banach algebra of bounded linear operators on $X$.   For $\l\in\C$, we will let $r_\l$ denote the function $r_\l(z) := (\l+z)^{-1}$ with appropriate domain in $\C$.  Let $\A$ be a Banach algebra such that $\B$ is continuously included in $\A$ and $\A$ is continuously included in $H^\infty(\C_+)$.     We say that $A$ {\it has a bounded $\A$-calculus} if there is a bounded algebra homomorphism $\Phi : \A \to L(X)$ such that $\Phi(r_\l) = (\l + A)^{-1}$ for some (or equivalently, all) $\l\in\C_+$.

It is shown in \cite[Theorem 4.4]{BGT} and \cite[Theorem 6.1]{BGT2} that $A$ satisfies the (GSF) condition if and only if $A$ has a bounded $\B$-calculus.  Moreover, $A$ satisfies the (GSF) condition if $\Phi$ is defined and bounded only on $\B_{00}$, as all the functions $G_{\a,\varphi}$ in the proof of \cite[Theorem 6.1]{BGT2} belong to $\B_{00}$.   In addition, any bounded $\B$-calculus for $A$ is unique \cite[Theorem 6.2]{BGT2}, and we will denote it by $\Phi_\B^A$.

In the case when $\A=\B$, the following result coincides with Theorem~\ref{KTB}, since $\B_{00}$ is the closure of $\LL$ in $\B$.   In general, the closure of $\LL$ in $\A$ is contained in $\{f \in\A: \lim_{z\in\C_+,\, |z|\to\infty} f(z) = 0 \}$, but equality may not hold.

\begin{theorem} \label{KTA}
Let $\A$ be a Banach algebra such that $\B$ is continuously included in $\A$ and $\A$ is continuously included in $H^\infty(\C_+)$.  
Let $-A$ be the generator of a bounded $C_0$-semigroup $(T(t))_{t\ge0}$ and assume that $A$ has a bounded $\A$-calculus.   Let $f$ be in the closure of $\LL$ in $\A$, and assume that $f$ vanishes on $\sigma(A) \cap i\R$.  Then
$\lim_{t\to\infty} \|T(t)f(A)\| = 0$.
\end{theorem}

To prove Theorem \ref{KTA}, we will need to consider  the case of a {\it skew-hermitian} operator $Z$ on a Banach space $Y$, so that $-Z$ generates a $C_0$-group of isometries on $Y$.   Then $\sigma(Z) \subseteq i\R$.   If $Z$ has a bounded $\B$-calculus,  then $Z$ also has a bounded $C_0(i\R)$-calculus, that is, a bounded algebra homomorphism from $C_0(i\R)$ to $L(X)$ mapping $r_\l$ to $(\l+Z)^{-1}$ for some (or all) $\l \in \C \setminus i\R$.  The converse also holds.  These statements can be seen in \cite[Theorem 6.5]{BGT2}, or by using a combination of \cite[Theorem 6.1]{BGT2}, \linebreak[2] \cite[Lemma~3.4]{CG08} and \cite[Theorem 3.6]{dLK95}.   The density of the algebra generated by the resolvent functions $\{r_\l : \l \in \C \setminus i\R\}$ in $C_0(i\R)$ shows that any bounded $C_0(i\R)$-calculus for $Z$ is unique, and we denote it by $\Phi_{C_0}^Z$.   If $f \in \B_{00}$, then $\Phi_\B^Z(f)$ and $\Phi_{C_0}^Z(f)$ are both defined (here we use $f$ to denote the boundary function of $f$ on $i\R$).  The proof of \cite[Theorem 6.5]{BGT2} did not show explicitly that the two definitions coincide.   This compatibility can be deduced from the paragraph following the proof of Theorem 6.5 and equation (6.4) in \cite{BGT2}.  They establish a multiplier formula for $\Phi_\B^Z$, namely
\begin{equation} \label{mult}
\Phi_\B^Z(f) \Phi_{C_0}^Z(g) = \lim_{n\to\infty} \Phi_{C_0}^Z(fe_n) \Phi_{C_0}^Z(g)
\end{equation}
in the strong operator topology, where $g \in C_0(i\R)$ and $(e_n)_{n\ge1}$ is a bounded approximate unit in $C_0(i\R)$.   The same formula holds when $\Phi_\B^Z(f)$ is replaced by $\Phi_{C_0}^Z(f)$.   The compatibility of the two calculi follows; see also Remark~\ref{rem1}.   

If $Z$ satisfies the (GSF) condition, then $-Z$ also satisfies the (GSF) condition.   This can be seen from \cite[Lemma 3.4]{CG08}.   It corresponds to the fact that if $Z$ has a bounded $C_0(i\R)$-calculus, then so does $-Z$, with $\Phi_{C_0}^{-Z}(f) = \Phi_{C_0}^Z(\tilde f)$, where $\tilde f(i\b) = f(-i\b)$ for $f \in C_0(i\R)$ and $\b\in\R$.

\section{The proof of Theorem \ref{KTA}}  \label{proof}

The first step in the proof is the observation that it suffices to prove the result under the additional assumption that $\B$ is dense in $\A$.    We make this assumption without any loss of generality, as $\A$ can be replaced by the closure of $\B$ in $\A$.   This assumption implies that all functions in $\A$ are bounded and uniformly continuous on $\overline{\C_+}$ and any bounded $\A$-calculus for an operator $A$ is unique, by the uniqueness of any $\B$-calculus and the density of $\B$ in $\A$.    Throughout this section, we take $\A$ to be a fixed Banach algebra of this form.   It may be advantageous to take $\A=\B$ on first reading of the proof.

Next we consider the case of a skew-hermitian operator $Z$ which has a bounded $\B$-calculus.  Then $Z$ also has a bounded $\A$-calculus $\Phi_\A^Z$, given by the multiplier formula \eqref{mult} with $\B$ replaced by $\A$.   This calculus is an extension of the $\B$-calculus for $Z$, and it is compatible with the $C_0(i\R)$-calculus for $Z$.

Now assume that $Z$ is a bounded skew-hermitian operator with a bounded $\B$-calculus, and let $K$ be a compact subset of $i\R$ such that $\sigma(Z)$ is contained in the interior of $K$ in $i\R$.   The following lemma shows that there is a bounded $C(K)$-calculus for $Z$ with properties similar to those of a single polynomially bounded isometry; see \cite[Lemma 1.1]{KvN97}.   A {\it bounded $C(K)$-calculus} for $Z$ is a bounded algebra homomorphism mapping the constant function 1 to the identity operator $I$ and the identity function 
  to $Z$, or equivalently mapping $r_\l$ to $(\l+Z)^{-1}$ for all $\l \in \C\setminus K$.  Such a calculus is necessarily unique by the density of the polynomials in $C(K)$.

\begin{lemma} \label{lem1}
Let $Z$ be a bounded skew-hermitian operator on a Banach space $Y$, and assume that $Z$ has a bounded $\B$-calculus. Let $\A$ be a Banach algebra as described above, and let $K$ be a compact subset of $i\R$ such that $\sigma(Z)$ is contained in the interior of $K$ in $i\R$.  There is a unique bounded $C(K)$-calculus $\Psi$ for $Z$.   It has the following properties:
 \begin{enumerate}[\rm(a)]
 \item If $f\in C_0(i\R)$ and $g$ is the restriction of $f$ to $K$, then $\Psi(g)=\Phi_{C_0}^Z(f)$.
\item If $f \in \A$ and $g$ is the restriction of $f$ to $K$, then $\Psi(g) = \Phi_\A^Z(f)$.
\item If $g \in C(K)$, then $\sigma(\Psi(g)) = g(\sigma(Z))$.
\item If $g \in C(K)$ and $g$ vanishes on $\sigma(Z)$, then $\Psi(g)= 0$.
\end{enumerate}
\end{lemma}

\begin{proof}
Since $Z$ has a bounded $\B$-calculus, the (GSF) condition holds for $Z$ and for $-Z$, so,  for all $\a>0$, $y \in Y$ and $y^* \in Y^*$,
\[
\a \int_\R \left| \langle (\a+i\b-Z)^{-2}y,y^* \rangle \right| \,d\b \le \gamma_{-Z}\, \|y\|\,\|y^*\|.
\]

For $\a \in (0,1]$ and $\b\in\R$, let
\begin{align*}
\Delta(\a,\b,-Z) &:= (\a+i\b-Z)^{-1} - (-\a+i\b-Z)^{-1} \\
&= - 2\a (\a+i\b-Z)^{-1} (-\a+i\b-Z)^{-1}.
\end{align*}
Fix $k > \|Z\|$, and let $K' = \{\b\in\R: i\b\in K\}$.   Then
\begin{align*}
\|\Delta(\a,\b,-Z)\| &\le \frac{2\a}{(|\b|-\|Z\|)^2}, \quad &&\alpha\in(0,1],\ |\b|>k, \\
\|\Delta(\a,\b,-Z)\| &\le c, \qquad &&\a\in(0,1], \ |\b|\le k, \; \b \in \R \setminus K',\\
\lim_{\a\to0+} \|\Delta(\a,\b,-Z)\| &= 0, \qquad &&i\b \notin \sigma(Z),
\end{align*}
for some constant $c$.   The Dominated Convergence Theorem and these estimates imply that 
\begin{equation} \label{R-K}
\lim_{\a\to0+} \int_{\R \setminus K'} \|\Delta(\a,\b,-Z)\| \,d\b = 0.
\end{equation}

Let $f \in C_0(i\R)$.  It follows from \cite[Theorem 3.6]{dLK95} and \eqref{R-K} that the $C_0(i\R)$-calculus for $Z$ is given by
\begin{align*}
\langle \Phi_{C_0}^Z(f) y,y^* \rangle &= \frac{1}{2\pi} \lim_{\a\to0+} \int_\R f(i\b) \langle \Delta(\a,\b,-Z)y,y^* \rangle \, d\b \\
&= \frac{1}{2\pi} \lim_{\a\to0+} \int_{K'} f(i\b) \langle \Delta(\a,\b,-Z)y,y^* \rangle \, d\b.
\end{align*}
 From \cite[Lemma 3.4]{CG08} applied to $-iZ$, there is a constant $C$ such that 
\begin{equation*} \label{cj3.6}
\int_\R \left| \langle \Delta(\a,\b,-Z)y,y^* \rangle \right| \, d\b \le C\,\|y\|\,\|y^*\|, \qquad \alpha\in(0,1],\ y \in Y,\ y^*\in Y^*.
\end{equation*}
   Then
\begin{align*}
\left| \langle \Phi_{C_0}^Z(f)y,y^* \rangle \right|  &\le \frac{1}{2\pi} \int_{K'} \|f\|_{C(K)} \left| \langle \Delta(\a,\b,-Z)y,y^* \rangle \right|\,d\b \\
& \le \frac{C}{2\pi} \,\|f\|_{C(K)}\, \|y\|\,\|y^*\|.
\end{align*}
This implies that
\begin{equation} \label{CKest}
\|\Phi_{C_0}^Z(f)\| \le \frac{C}{2\pi} \|f\|_{C(K)}.
\end{equation}

We can now define a bounded $C(K)$-calculus for $Z$, as follows.   Let $g \in C(K)$, and let $f \in C_0(i\R)$  be any extension of $g$.   Define $\Psi(g) := \Phi_{C_0}^Z(f)$.   It follows from \eqref{CKest} that this definition is independent of the choice of $f$, and $\Psi$ is a bounded algebra homomorphism from $C(K)$ to $L(Y)$ mapping $r_\l$ to $(\l+Z)^{-1}$ for $\l\in \C\setminus K$.  Thus it is a bounded $C(K)$-calculus for $Z$.    Property (a) in Lemma \ref{lem1} is immediate from the definition of $\Psi$.

The map from $\A$ to $L(Y)$ given by  $f \mapsto \Psi(f|_K)$ is a bounded algebra homomorphism, and it is a bounded $\A$-calculus for $Z$.  Now (b) follows from the uniqueness of the $\A$-calculus for $Z$; see the first paragraph in Section~\ref{proof}.

\def\Ca{\mathcal{C}}

The proofs of (c) and (d)  are very similar to the proofs of \cite[Lemma~1.1]{KvN97}.   For (c), we may take any commutative Banach subalgebra $\Ca$ of $L(Y)$ containing $I$, $Z$, $\Psi(g)$ and all their resolvents.  Then the spectra of $Z$ and $\Psi(g)$ in $L(Y)$ coincide with their spectra in $\Ca$, so it suffices to show that $\chi(\Psi(g)) = g(\chi(Z))$  for every character $\chi$ of $\Ca$.   We may approximate $g$  uniformly on $K$ by a sequence $(p_n)_{n\ge1}$ of polynomials, so $\Psi(g) = \lim_{n\to\infty} p_n(Z)$ in operator norm.   Then
\[
\chi(\Psi(g)) = \lim_{n\to\infty} \chi(p_n(Z)) = \lim_{n\to\infty} p_n(\chi(Z))  = g(\chi(Z)).
\]
This establishes (c).

For (d), first assume that  $g \in C(K)$ vanishes on a neighbourhood of $\sigma(Z)$ in $K$.  Choose $h \in C(K)$ such that  $h(\l)=1$ for all $\l\in\sigma(Z)$ and $gh=0$.    By (c), $\sigma(\Psi(h)) = \{1\}$, so $\Psi(h)$ is invertible.    Since $\Psi(g)\Psi(h)=0$ it follows that $\Psi(g) = 0$.  

Now assume only that $g \in C(K)$ vanishes on $\sigma(Z)$.  Then there is a sequence $(g_n)_{n\ge1}$ in $C(K)$ where each $g_n$ vanishes on a neighbourhood $U_n$ of $\sigma(Z)$ in $K$, and $\|g_n-g\|_{C(K)}\to0$.  Statement (d) follows from \eqref{CKest} and the paragraph above.
\end{proof}

Properties (a) and (b) in Lemma \ref{lem1} establish that the functional calculus $\Psi$ for a bounded skew-hermitian operator $Z$ is compatible with the $C_0(i\R)$-calculus and the $\A$-calculus for functions $f \in \A$.  We will use the simple notation $f(Z)$ for the resulting operators in any of these calculi.

We now return to the case when $Z$ is an unbounded skew-hermitian operator with a bounded $\B$-calculus.

\begin{lemma}  \label{lem2}
Let $Z$ be a skew-hermitian operator on a Banach space $Y$, and assume that $Z$ has a bounded $\B$-calculus.  Let $f \in \A$, 
and assume that $f$ vanishes on $\sigma(Z)$.  Then $f(Z) = 0$. 
\end{lemma}

\begin{proof}
We will apply the Arveson spectral theory for $C_0$-groups of isometries to the $C_0$-group $(V(t))_{t\in\R}$ generated by $-Z$ on $Y$; see \cite[Section 8]{Dav80} for an account of the theory in this context. 

For $k\in\N$, let $Y_k$ be the spectral subspace corresponding to the interval $[-k,k]$.   Then $Y_k$ is a closed subspace of $Y$ which is invariant under the operators $V(t)$, and the restrictions $V_k(t)$ of $V(t)$ to the subspace $Y_k$ form a norm-continuous group of isometries on $Y_k$.  Moreover, the negative generator $Z_k$ is a bounded operator on $Y_k$, so $Y_k$ is contained in the domain of $Z$, and $Zy =Z_k y$ for all $y \in Y_k$, and $\sigma(Z_k) \subseteq \sigma(Z) \cap i[-k,k]$; see \cite[Theorems 8.19 and 8.27]{Dav80}.   By restricting the operators $f(Z)$ to the subspace $Y_k$ for $f \in C_0(i\R)$, we obtain a bounded $C_0(i\R)$-calculus for $Z_k$.

Now let $f \in \A$, 
and assume that $f$ vanishes  on $\sigma(Z)$.   For each $k$, $f$ vanishes on $\sigma(Z_k)$.   Taking $K = i[-(k+1),k+1]$, Lemma \ref{lem1} shows that $f(Z_k) = 0$.    This implies that $f(Z)y = 0$ for all $y \in Y_k$, since $f(Z_k)$ is the restriction of $f(Z)$ to $Y_k$.   Since $\bigcup_{k\in\N} Y_k$ is dense in $Y$ (see \cite[Lemma 8.12]{Dav80}), this implies that $f(Z)=0$.
\end{proof}

\begin{remark} \label{rem1}
Let $Z$ be a skew-hermitian operator on $Y$, with a bounded $\B$-calculus, and let $Y_k$ and $Z_k$ be as in the proof of Lemma \ref{lem2}.   Let $f \in \A \cap C_0(i\R)$.   Parts (a) and (b) of Lemma \ref{lem1} show that $\Phi_\A^{Z_k}(f) = \Phi_{C_0}^{Z_k}(f)$.   Hence $\Phi_\A^Z(f)$ and $\Phi_{C_0}^Z(f)$ coincide on $\bigcup_{k\in\N} Y_k$, and then by continuity they coincide on $Y$.   This is an alternative proof that the two calculi are compatible.
\end{remark}

We now prove Theorem \ref{KTA}.   The structure of our argument is the same as that of \cite{Vu92}, but we need the lemmas above to justify the crucial stage in the argument which enables us to cover functions outside $\LL$ and to avoid assumptions of spectral synthesis.

\begin{proof} [Proof of Theorem \ref{KTA}]
If $i\R \subseteq \sigma(A)$, $f\in\A$, and $f$ vanishes on $i\R$, then $f$ vanishes on $\C_+$.  So $f(A) = 0$, and the result holds trivially.   Thus we may assume that $i\R$ is not contained in $\sigma(A)$.   We will initially show that $\lim_{t\to\infty} T(t)f(A) = 0$ in the strong operator topology.

We use the limit isometric semigroup, as in \cite[Theorem 3.2]{Vu92} or  \cite[Proposition 3.1]{Vu97}.   There is a Banach space $Y$ (which may be $\{0\}$), a bounded map $\pi : X \to Y$ with dense range such that $\|\pi(x)\|_Y =\limsup_{t\to\infty} \|T(t)x\|$ for $x \in X$, and a $C_0$-semigroup  $(V(t))_{t\ge0}$ of (not necessarily invertible) isometries on $Y$ such that $V(t) \pi = \pi T(t)$ and the negative generator $Z$ of $(V(t))_{t\ge0}$ satisfies $\sigma(Z) \subseteq \sigma(A) \cap i\R$.    

By a property of semigroups of isometries (see \cite[Lemma, p.~38]{LV88} or \cite[pp.~419, 420]{GS78}), the inequality $\|Zy - \l y\| \ge \Re\l \|y\|$ is valid for all $\l\in\C_+$ and $y \in Y$.  Since $i\R$ is not contained in $\sigma(Z)$ and $Z$ has no approximate eigenvalues in $\C_+$, it follows that  $\sigma(Z) \cap \C_+$ is empty, and \linebreak[4] $\|(\l-Z)^{-1}\| \le (\Re\l)^{-1}$ for $\l\in\C_+$.  By the Hille-Yosida theorem, $Z$ generates a $C_0$-semigroup, and hence $-Z$ generates a $C_0$-group of invertible isometries $(V(t))_{t\in\R}$ on $Y$.  

For $g \in \B$, the operator $g(A) \in L(X)$ commutes with $T(t)$ for all $t\ge0$, and so there is a unique operator $\Upsilon(f) \in L(Y)$, such that $\Upsilon(f)\pi = \pi f(A)$.  Then $\Upsilon$ is a bounded algebra homomorphism of $\B$ into $L(Y)$ mapping $r_\l$ to $(\l+Z)^{-1}$ for $\l\in\C_+$, so it is a bounded $\B$-calculus for $Z$.  

Now assume that $f$ is in the closure of $\LL$ in $\A$ and $f$ vanishes on $\sigma(A) \cap i\R$.   Then $f$ vanishes on $\sigma(Z)$, and 
it follows from Lemma \ref{lem2} that $\Upsilon(f)=0$.   Then  $\pi f(A) = 0$, and hence $\lim_{t\to\infty} T(t)f(A) = 0$ in the strong operator topology.

It remains to lift this conclusion to the operator-norm topology.   This can be achieved by using the method of \cite[Theorem 3.2]{Vu92} and \cite[Theorem~3.9]{Vu97}.  Consider the induced bounded $C_0$-semigroup of left multiplications by $T(t)$ on the Banach space $\mathcal{X}$ of operators $S \in L(X)$ such that $t \mapsto T(t)S$ is norm-continuous on $\R_+$.   Let $H$ be the negative generator  of this $C_0$-semigroup on $\mathcal{X}$.  It is easily seen  that $\sigma(H) \subseteq \sigma(A)$, and that the operators of left multiplication by $f(A)$, for $f \in \A$, form an $\A$-calculus for $H$.  It is well known and elementary that $g(A) \in \mathcal{X}$ for all $g \in \LL$.     Applying the result established in the paragraph above to the $C_0$-semigroup on $\mathcal{X}$ shows that $\lim_{t\to\infty} \|T(t)f(A)g(A)\| = 0$ for all $g\in\LL$.

Let $(e_n)_{n\ge1}$ be a bounded approximate unit for $L^1(\R_+)$, and let $g_n = \mathcal{L}e_n$ for $n\ge1$.  Then $(g_n)_{n\ge1}$ is a bounded approximate unit for $\mathcal{L}L^1$ and hence for the closure of $\LL$ in $\A$; see Section~\ref{sec2}.    Thus $\lim_{n\to\infty} \|fg_n - f\|_\B = 0$,  and so $\lim_{n\to\infty} \|f(A)g_n(A)- f(A)\|=0$.  Since the semigroup $(T(t))_{t\ge0}$ is bounded and $\lim_{t\to\infty} \|T(t)f(A)g_n(A)\|=0$ for all $n\ge1$, it follows that $\lim_{t\to\infty} \|T(t)f(A)\|=0$. 
\end{proof}

\begin{rems}
(a) The above proof, without the final paragraph,  shows that  $\lim_{t\to\infty} T(t)f(A) = 0$ in the strong operator topology and $$\lim_{t\to\infty} \|T(t)f(A)g(A)\|=0$$ for all $g \in \LL$ even without the assumption that $f$ vanishes at infinity.

\smallskip
\noindent (b)  If  $\sigma(A) \cap i\R$ is bounded then the application of Lemma \ref{lem2} in the proof of Theorem \ref{KTA} may be replaced by an application of Lemma~\ref{lem1}.
\smallskip

\noindent (c) In Theorem \ref{KTB} where $\A=\B$ and $f\in\B$, the condition that $f$ vanishes on $\sigma(A)\cap i\R$ is necessary for the conclusion. Indeed, let $-A$ be the generator of a bounded $C_0$-semigroup $(T(t))_{t\ge0}$ and assume that $A$ has a bounded $\B$-calculus.   Let $f \in \B$, and suppose that  $\lim_{t\to\infty} \|T(t)f(A)\| = 0$.   Let $z\in \sigma(A)\cap i\R$. Then $e^{ tz}f(z)\in\sigma(T(t)f(A))$ for all $t\ge0$ by the spectral inclusion theorem for the $\B$-calculus \cite[Theorem~4.17]{BGT}, and hence $|f(z)|\le \|T(t)f(A)\|\to0$ as $t\to\infty$. It follows that $f(z)=0$, as required.
\smallskip

\noindent (d)   In Theorem \ref{KTA}, we have assumed that $\A$ contains $\B$, and so contains the constant functions.    It would suffice to assume that $\A$ contains the subalgebra $\B_0$ defined by
\[
\B_0 := \left\{f \in \B:  \lim_{\Re z \to \infty} f(z) = 0 \right\},
\]
as one may pass to the algebra obtained by adjoining the constant functions.  
\end{rems}

\section{Examples}  \label{exs}

Here we briefly describe two Banach algebras introduced by Arnold and Le Merdy, to which Theorem \ref{KTA} can be applied.  
In \cite{ALM} they introduce a Banach algebra $\A(\C_+)$ which is continuously included in $H^\infty(\C_+)$, they identify a Banach subalgebra $\A_0(\C_+)$ and they show that it is the closure of the algebra $\LL$ in $\A(\C_+)$ \cite[Lemma 3.14]{ALM}.   They also show that $\B_0$ is properly and continuously included in $\A(\C_+)$ \cite[Proposition 5.2, Theorem 5.3]{ALM}.    (Readers should be aware that our spaces $\B_0$ and $\B_{00}$ are denoted  by $\B(\C_+)$ and $\B_0(\C_+)$, respectively, in \cite{ALM}.)   In finding the Banach algebra $\A(\C_+)$ with these properties, the authors were guided by Peller's work in the discrete case \cite{Pel82}.   It is not easy to identify specific functions which are in $\A(\C_+)$ but not in $\B$.

In a further paper with an additional author \cite{ALMZ}, the authors introduce a Banach algebra $\A_S(\C_+)$ such that
\[
\A(\C_+) \subsetneq \A_{S}(\C_+) \subsetneq H^\infty(\C_+),
\]
with continuous inclusions.   The closure of $\LL$ in $\A_{S}(\C_+)$ is an identified subalgebra $\A_{0,S}(\C_+)$.    Again the definitions are complex, and finding explicit functions is not easy. 

Let $-A$ be the generator of a bounded $C_0$-semigroup on a Hilbert space.   The authors show in \cite{ALM} that $A$ has a (unique) bounded $\A(\C_+)$-calculus, and in \cite{ALMZ} that $A$ has a bounded $\A_S(\C_+)$-calculus.   Hence Corollary \ref{KTH} can be extended to the following.

\begin{corollary}
Let $-A$ be the generator of a bounded $C_0$-semigroup $(T(t))_{t\ge0}$ on a Hilbert space. Let $f \in \A_{0,S}(\C_+)$, and assume that $f$ vanishes on $\sigma(A) \cap i\R$.  Then $\lim_{t\to\infty} \|T(t)f(A)\| = 0$.
\end{corollary}



\end{document}